\documentclass[11pt]{amsart}
\usepackage{amsmath,amssymb,amsfonts,graphics,geometry,url,hyperref}

\DeclareMathOperator{\SL}{SL}

\DeclareMathOperator{\meas}{meas}
\newcommand{\A}{\mathcal{A}}

\newcommand{\F}{\mathcal{F}}

\newcommand{\h}{\mathcal{H}}
\newcommand{\N}{\mathbb{N}}
\newcommand{\R}{\mathbb{R}}
\newcommand{\Z}{\mathbb{Z}}
\newcommand{\eps}{\varepsilon}
\newtheorem{thm}{Theorem}
\newtheorem{alg}[thm]{Algorithm}
\newtheorem{conj}[thm]{Conjecture}
\newtheorem{Def}[thm]{Definition}
\newtheorem{lem}[thm]{Lemma}
\newtheorem{rem}[thm]{Remark}

\begin{document}
\title[Maass forms and fluctuations in the Weyl remainder]{Large sets of
  consecutive Maass forms and fluctuations in the Weyl remainder}
\author[H.~Then]{Holger Then}
\address{University of Bristol, Department of Mathematics, University Walk,
  Bristol BS8 1TW, United Kingdom.}

\begin{abstract}
  We explore an algorithm which systematically finds all discrete
  eigenvalues of an analytic eigenvalue problem.
  The algorithm is more simple and elementary as could be expected before.
  It consists of Hejhal's identity, linearisation, and Turing bounds.
  Using the algorithm, we compute more than one hundredsixty thousand
  consecutive eigenvalues of the Laplacian on the modular surface,
  and investigate the asymptotic and statistic properties of the
  fluctuations in the Weyl remainder.
  We summarize the findings in two conjectures.
  One is on the maximum size of the Weyl remainder, and the other is on the
  distribution of a suitably scaled version of the Weyl remainder.
\end{abstract}

\thanks{The author thanks Andrew Booker, Sally Koutsoliotas,
  Stefan Lemurell, and Andreas Str\"ombergsson for organizing
  the AIM workshop {\em Computing Arithmetic Spectra}, March 2008.
  This event raised essential ideas for my current research.
  Further thanks are to Andrew Booker and Andreas Str\"ombergsson for
  useful discussions and for sharing their work in preparation.
  The author is supported by EPSRC Fellowship EP/H005188/1.}

\maketitle

\section{Introduction}

The real-analytic eigenvalue problem of the Laplacian on
finite, non-compact hyperbolic surfaces is an important one.
The solutions are automorphic forms and play a central role in
analytic number theory.
They can be used to extend the classical theory of Dirichlet series
with Euler products \cite{Maass1949}
and are closely related to the Millennium Problem of the Riemann
Hypothesis \cite{Sarnak2004}.
Moreover, they serve to find class numbers of real quadratic number fields
and hence solve questions which already inspired Gauss \cite{Booker+}.
Non-holomorphic automorphic forms, the so called Maass forms, are
intensively studied in spectral theory.
The spectral decomposition into Maass forms let to the discovery
of Selberg's trace formula \cite{Selberg1956,Hejhal1983}
which connects the spectrum to the geometric properties of the
underlying space.
Automorphic forms are the prime example of quantum unique ergodicity
\cite{Lindenstrauss2006}.
In addition, Maass forms span a Hilbert basis in quantum mechanics on
hyperbolic surfaces and serve as important examples in quantum chaos.
For instance, they play a distinguished role in eigenvalue statistics
\cite{BogomolnyGeorgeotGiannoniSchmit1992,BolteSteilSteiner1992}.
As a complete set of eigenfunctions of the Laplacian on hyperbolic
surfaces they have also found applications in general relativity and
cosmology \cite{AurichLustigSteinerThen2004}.

By approximating them, numerics can bring Maass forms closer to us
\cite{HejhalRackner1992}.
It is clear how to compute Maass forms \cite{Hejhal1999}.
However, it was difficult to find {\em consecutive} sets of solutions.
We present an algorithm which allows us to find large consecutive
sets of Maass forms efficiently.
The algorithm is based on three ingredients:
Hejhal's identity to compute eigenfunctions corresponding to given
eigenvalues;
linearisation which converts the analytic
eigenvalue problem locally to a matrix eigenvalue problem;
and explicit Turing bounds which serve to check and complete the results.

As they lie on our route, we apply results of ergodic theory, namely
the equidistribution of long closed horocycles
\cite{Zagier1981,Sarnak1981,Hejhal1996,Hejhal2000,FlaminioForni2003,
  Strombergsson2004}.
This allows us to unreveil an elegant view on the derivation of Hejhal's
identity in section \ref{Hejhal}, where we present the computation of
Maass forms under the preliminary assumption that the discrete eigenvalues
would be known.
The eigenvalue problem is linearised in section \ref{Linear},
thereby almost all of the eigenvalues are found.
Section \ref{Turing} presents Turing bounds.
These bounds drive a control algorithm which systematically checks and
completes the list of Maass forms until a large set of
{\em consecutive} Maass forms is found.
We have computed more than $160\,000$ consecutive Maass forms.
This exceeds any previous numerical solution of any non-integrable
system by far.
Results are listed in section \ref{Results}, where we investigate the
asymptotic and statistic properties of the fluctuations in the Weyl
remainder.

\section{Fundamental domains, equidistribution of long closed horocycles,
  and congruent points}\label{Then}

In this section, we will consider fundamental domains,
present a pullback algorithm, and use ergodic properties to
show that a discrete subgroup of the isometries can
be completely specified by a set of congruent pairs of points.

Let $X=\Gamma\backslash\h$ be a surface, where $\Gamma$ is a
cofinite, non-cocompact subgroup of $\SL(2,\R)$ which acts properly
discontinuous on the Poincar\'e upper half-plane $\h$.
The action is given by linear fractional transformations.
On $X$ we have the invariant metric $ds=|dz|/y$.
We use $\ell(\cdot)$ to denote length, and
$d(\cdot,\cdot)$ to denote hyperbolic distance.

Because $\Gamma$ is non-cocompact, $X$ has at least one cusp.
After a suitable coordinate transformation we may assume that
one of the cusps lies at $\infty$, and that the corresponding
isotropy subgroup $\Gamma_\infty\subset\Gamma$ is generated by
$z\mapsto z+1$.

Take some point $p$ in $\h$, not an elliptic fixed point.
Let $\F$ be the corresponding Dirichlet fundamental domain:
\begin{align*}
  \F=\{z\in\h\ |\ d(p,z)\le d(p,\gamma z)\ \forall\gamma\in\Gamma\}.
\end{align*}
We know from Siegel's theorem that $\F$ has a finite number of sides.
The sides of $\F$ fall into congruent pairs, let us
denote them $s_1,s_{n+1};s_2,s_{n+2};\ldots;s_n,s_{2n}$, and let
$g_1,g_2,\ldots,g_n\in\Gamma$ be the identification maps,
$g_ks_k=s_{n+k}$.
(Recall the convention that an elliptic fixed point of order two is
considered as a vertex separating {\em two distinct sides} of $\F$.)
We know that $g_1,g_2,\ldots,g_n$ generate $\Gamma$
(see \cite[pp.~73--74]{Katok1992}).

The following is an algorithm for computing the pullback of any given
point in the hyperbolic plane into a Dirichlet fundamental domain.
\begin{alg}[Pullback algorithm \cite{Strombergsson2000}]\label{alg:pullback}
  \begin{enumerate}
  \item Start with $z\in\h$.
  \item\label{alg:pullback.2} Compute the $2n$ points
    $g_1z,g_1^{-1}z,g_2z,g_2^{-1}z,\ldots,g_n^{-1}z$.
    Let $z'$ be the one of these which has {\em smallest} hyperbolic
    distance to $p$.
  \item If $d(p,z')<d(p,z)$, then {\em replace $z$ by $z'$}, and
    repeat from \ref{alg:pullback.2}.
  \item In the other case, $d(p,z')\ge d(p,z)$, we {\em know that $z$
    lies in $\F$}, hence $z$ is the searched-for pullback.
  \end{enumerate}
\end{alg}
Str\"ombergsson proved that his pullback algorithm always finds the pullback
within a finite number of operations \cite{Strombergsson2000}.

We use $z^*=x^*+iy^*$ to denote the pullback of $z=x+iy$,
and $\A^*=\{z^*\ |\ z\in \A\}$ to denote the pullback of $\A$,
for any $\A\subseteq\h$.

For any $y>0$, the curve $L_y=\{x+iy\ |\ x\in[0,1]\}$ is a closed
horocycle of length $\ell=1/y$ on $X$.
When $\ell\to\infty$, this curve is known to become uniformly
equidistributed on $X$ with respect to the Poincar\'e area
$d\mu=dxdy/y^2$ \cite{Strombergsson2004}.
Let us consider $Q$ equidistant points $\{z_j\}_{1\le j\le Q}$
on the closed horocyle $L_y$.
Using Algorithm \ref{alg:pullback}, we can compute
$\{z_j^*\}_{1\le j\le Q}$, and construct the
congruent pairs of points $\{(z_j,z_j^*)\}_{1\le j\le Q}$.

Of importance is:
\begin{lem}\label{lem:z,z}
  Fix $Q\in\N$ and $y>0$, and consider $Q$ equidistant points
  $\{z_j\}_{1\le j\le Q}$ on the closed horocycle $L_y$.
  For $Q$ sufficiently large and $y$ sufficiently small,
  the congruent pairs of points $\{(z_j,z_j^*)\}_{1\le j\le Q}$
  contain all information of the group $\Gamma$.
\end{lem}

For $\A\subseteq\h$, we use $\partial\A$ to denote its boundary, and
$\Gamma\A=\{\gamma z\ |\ z\in\A,\gamma\in\Gamma\}$ to denote its orbit.

For the proof of Lemma \ref{lem:z,z}, we need the following:
\begin{lem}\label{lem:l&s}
  If the horocycle $L_y^*$ comes sufficiently close to each point
  in $\F$, then we have:
  For each $1\le k\le n$ there exists a line segment $l_k\subset L_y$
  such that $l_k$ meets $\Gamma\partial\F$ exactly once,
  $l_k^*$ does not meet $\partial\F$ in a vertex, and
  $\ell(l_k)>\ell(l_k\cap\gamma_k^{-1}\F)>0$ for some $\gamma_k\in\Gamma$.
\end{lem}
\begin{proof}
Fix $1\le k\le n$.
Fix some point $\omega_k\in s_k$, not a vertex of $\F$.
By assumption $L_y^*$ comes sufficiently close to any point in $\F$.
Thus, for some $\eps>0$ there exists $\gamma_k\in\Gamma$ such that
$|y-\Im(\gamma_k^{-1}\omega_k)|<\eps$.
If unexpectedly
$\Im(\gamma_k^{-1}\omega_k)=\sup_{\omega\in s_k}\Im(\gamma_k^{-1}\omega)$
should hold, we repeat the proof with another $\omega_k\in s_k$.

We have
$\inf_{\omega\in s_k}\Im(\gamma_k^{-1}\omega)<\Im(\gamma_k^{-1}\omega_k)<y+\eps$
and
$y-\eps<\Im(\gamma_k^{-1}\omega_k)<\sup_{\omega\in s_k}\Im(\gamma_k^{-1}\omega)$.
For $\eps$ sufficiently small, $L_y$ intersects $\gamma_k^{-1}s_k$
at least once (and at most twice) not at a vertex.

Take an open line segment $l_k\subset L_y$ such that
$l_k$ intersects $\gamma_k^{-1}s_k$,
and such that it meets $\Gamma\partial\F$ exactly once,
but not in a vertex of $\gamma_k^{-1}\F$.
Since $l_k$ is an open line segment, it has positive
length on both sides of the intersection with $\gamma_k^{-1}s_k$.
\end{proof}
\begin{proof}[Proof of Lemma \ref{lem:z,z}]
Our task is to extract the side identification maps
$g_1,\ldots,g_n$ from the pairs of points
$\{(z_j,z_j^*)\}_{1\le j\le Q}$.

The points $z_1,\ldots,z_Q$ are on the curve $L_y$.
In the limit $y\to0$, the closed horocycle $L_y$ equidistributes on
$\Gamma\backslash\h$.
If $y>0$ is finite, the curve $L_y^*$ is no longer dense in $\F$.
Still, for $y$ {\em sufficiently small}, $L_y^*$ comes sufficiently close
to each point in $\F$.

Fix some $1\le k\le n$.
Let $l_k$ be the line segment given by Lemma \ref{lem:l&s}.
There are two distinct elements $\gamma_k$,$\gamma_{n+k}\in\Gamma$,
$\gamma_k\not=\pm\gamma_{n+k}$, such that $\ell(l_k\cap\gamma_k^{-1}\F)>0$
and $\ell(l_k\cap\gamma_{n+k}^{-1}\F)>0$.
For $Q$ sufficiently large, at least three succesive points of
$\{z_j\}_{1\le j\le Q}$ are on $l_k\cap\gamma_k^{-1}\F$ and another three
are on $l_k\cap\gamma_{n+k}^{-1}\F$.
This is true for any $1\le k\le n$.
Hence, for each $1\le k\le n$ there are six succesive pairs of points
in $\{(z_j,z_j^*)\}_{1\le j\le Q}$ such that three of them satisfy
$z_j^*=\gamma_kz_j$, while the other three satisfy $z_j^*=\gamma_{n+k}z_j$.

{\em Now, let only the congruent pairs of points
  $\{(z_j,z_j^*)\}_{1\le j\le Q}$ be given.}
We search for all sequences of six succesive pairs of points in
$\{(z_j,z_j^*)\}_{1\le j\le Q}$ subject to the condition that three of them
satisfy $z_j^*=\gamma_hz_j$, while the other three of them satisfy
$z_j^*=\gamma_{h'}z_j$, where $\gamma_h,\gamma_{h'}\in\SL(2,R)$,
$\gamma_h\not=\pm\gamma_{h'}$.
By construction all $\gamma_h$ and $\gamma_{h'}$ are in $\Gamma$.
Moreover, amongst all these sequences of six succesive pairs of points,
for each $1\le k\le n$ there is at least one sequence which
is associated to the congruent pair of sides $s_k,s_{n+k}$.
The corresponding side identification map (or its inverse) follows from
$g_k=\gamma_h\gamma_{h'}^{-1}$.
\end{proof}

\section{Maass forms}\label{Maass}

Let $X=\Gamma\backslash\h$ be a finite, non-compact surface
with invariant metric $ds=|dz|/y$.
According to the metric, the Laplacian reads
$\Delta=-y^2(\partial^2/\partial x^2+\partial^2/\partial y^2)$.
\begin{Def}[\cite{Maass1949}]
  A Maass form is a
  \begin{enumerate}
    \item real analytic, $f\in C^\infty(\h)$,
    \item square integrable, $f\in L^2(X)$,
    \item automorphic, $f(\gamma z)=f(z)\ \forall\gamma\in\Gamma$,
    \item eigenfunction of the Laplacian, $\Delta f(z)=\lambda f(z)$.
  \end{enumerate}
  If a non-constant Maass form vanishes in all cusps of $X$,
  it is called a Maass cusp form.
\end{Def}

According to the Roelcke-Selberg spectral resolution of the Laplacian
\cite{Selberg1956,Roelcke1966}, its spectrum contains both,
a discrete and a continuous part.
The discrete part of the spectrum is spanned by the constant
eigenfunction $f_0$ and a countable number of Maass cusp forms
$f_1,f_2,f_3,\ldots$ which we take to be ordered with increasing
eigenvalues $0=\lambda_0<\lambda_1\le\lambda_2\le\lambda_3\le\ldots$.
The continuous part of the spectrum $\lambda\ge1/4$,
is spanned by Eisenstein series.

In order to keep notation simple, we focus on Maass cusp forms
on finite surfaces with exactly one cusp.
We assume that the cusps lies at $\infty$, and that the corresponding
isotropy subgroup $\Gamma_\infty\subset\Gamma$ is generated by
$z\mapsto z+1$.

In this setting, Maass cusp forms associated to the eigenvalue
$\lambda=r^2+1/4$ have the Fourier expansion
\begin{align*}
  f(z)=\sum_{n\in\Z-\{0\}}a_ny^{1/2}K_{ir}(2\pi|n|y)e^{2\pi inx},
\end{align*}
where $K$ stands for the $K$-Bessel function which decays exponentially
for large arguments \cite[Eq.~(14)]{BookerStrombergssonThen}.
The expansion coefficients $a_n$ grow at most polynomially in $n$.
If we bound $y$ from below, $y\ge Y$, and allow for a small numerical
error of size $\eps>0$, there is an $M(Y)=const/Y$ such that
\begin{align}\label{eq:expansion}
  f(z)=\sum_{0\not=|n|\le M(Y)}a_ny^{1/2}K_{ir}(2\pi|n|y)e^{2\pi inx}
  +[|error|<\eps].
\end{align}
The constant in $M(Y)=const/Y$, which depends on $\eps>0$, $r$, and
the group $\Gamma$, can be worked out explicitly in each case.
Of importance is:
\begin{lem}
  Let the fundamental domain $\F$ be bounded from below,
  $Y_0:=\inf_{z\in\F}\Im(z)>0$.
  Apart from a numerical error of at most $\eps>0$,
  a Maass cusp form is completely specified by its
  eigenvalue and a finite number of expansion
  coefficients, $\{\lambda,\{a_n\}\}_{0\not=|n|\le M(Y_0)}$.
\end{lem}
\begin{proof}
  Let $\{\lambda,\{a_n\}\}_{0\not=|n|\le M(Y_0)}$ be given.
  By automorphy it is enough to know $f(z)$ inside the fundamental domain.
  There, we have $y\ge Y_0$, and the value of $f(z)$ follows from
  \eqref{eq:expansion}.
\end{proof}
\begin{rem}
  For any $\{\lambda,\{a_n\}\}_{0\not=|n|\le M(Y_0)}$, the Fourier
  expansion \eqref{eq:expansion} is real analytic, square integrable on $X$,
  vanishes in the cusp at $\infty$, and satisfies the eigenvalue equation
  of the Laplacian.
  However, {\em only for very specific choices of
    $\{\lambda,\{a_n\}\}_{0\not=|n|\le M(Y_0)}$},
  the corresponding function is automorphic, and hence a Maass cusp form.
\end{rem}

\section{Computing Maass forms}\label{Hejhal}

For the moment, we assume that the eigenvalue would be given.
Our task is to find the coefficients $\{a_n\}$ such that automorphy holds,
$f(\gamma z)=f(z)\ \forall\gamma\in\Gamma$.
We use the congruent pairs of points of Lemma \ref{lem:z,z} with
$0<y<Y_0$  and $Q\ge2M(y)$ to test automorphy, $f(z_j^*)=f(z_j)$ for all
$1\le j\le Q$.
This results in $Q$ linear equations in $2M$ unknowns
$\{a_n\}_{0\not=|n|\le M}$.

Note, however, that $f(z_j)=f(z_j^*)$ is ill-conditioned.
If one tries to solve for $a_M$, one gets
$a_{M}y^{1/2}K_{ir}(2\pi|M|y)e^{2\pi iMx_j}
=f(z_j^*)-[f(z_j)_{\text{less the $M$-th summand}}]$.
The absolute value of the l.h.s\ is bounded by $\eps$
which is smaller than the numerical error of the r.h.s..

We use exponentially weighted superpositions to convert the
ill-conditioned linear system of equations into a
well-conditioned linear system of equations \cite{Hejhal1999},
\begin{align*}
  \frac1Q\sum_{j=1}^Qf(z_j)e^{-2\pi imx_j}
  =\frac1Q\sum_{j=1}^Qf(z_j^*)e^{-2\pi imx_j}.
\end{align*}
Using geometric series on the l.h.s.\ and Fourier expanding the
r.h.s.\ results in Hejhal's identity:
\begin{multline}\label{eq:Hejhal}
  a_my^{1/2}K_{ir}(2\pi|m|y)+[|error|<\eps]
  =\frac1Q\sum_{j=1}^Qf(z_j)e^{-2\pi imx_j}\\
  =\frac1Q\sum_{j=1}^Qf(z_j^*)e^{-2\pi imx_j}
  =\sum_{0\not=|n|\le M(Y_0)}a_nV_{mn}+[|error|<\eps]
  \qquad\forall0\not=|m|\le M(y)
\end{multline}
where $V_{mn}
=\frac1Q\sum_{j=1}^Q{y_j^*}^{1/2}K_{ir}(2\pi|n|y_j^*)e^{2\pi i(nx_j^*-mx_j)}$.
\begin{alg}[Phase 2 \cite{Hejhal1999}]\label{alg:phase2}
  If $\{\lambda,\{a_n\}\}_{0\not=|n|\le M(Y_0)}$ are given, then any
  further coefficients follow {\em directly} from
  \begin{align*}
    a_m\simeq\frac1{y^{1/2}K_{ir}(2\pi|m|y)}\sum_{0\not=|n|\le M(Y_0)}V_{mn}a_n
    \qquad\forall|m|>M(Y_0)
  \end{align*}
  as $y\to0$.
  The numerical error is bounded by
  \begin{align*}
    |error\text{ of }a_m|<\frac{2\eps}{|y^{1/2}K_{ir}(2\pi|m|y)|}.
  \end{align*}
\end{alg}
We use $C(\lambda)$ to denote the matrix
$C_{mn}(\lambda)=\delta_{mn}-\frac{V_{mn}}{y^{1/2}K_{ir}(2\pi|m|y)}$,
$0\not=|m|\le M(Y_0)$, $0\not=|n|\le M(Y_0)$, and $\vec a$ to denote
the vector
$\vec a=(a_{-M(Y_0)},a_{-M(Y_0)+1},\ldots,a_{-1},a_1,a_2,\ldots,a_{M(Y_0)})$.
\begin{alg}[Phase 1 \cite{Hejhal1999}]\label{alg:phase1}
  If {\em only} $\lambda$ is given, then the coefficients
  $\{a_n\}_{0\not=|n|\le M(Y_0)}$ follow from {\em solving} the linear
  system of equations
  \begin{align*}
    C(\lambda)\vec a\simeq0, \qquad\vec a\not=0.
  \end{align*}
  The numerical error is bounded by
  \begin{align*}
    \Big|error\text{ of }\sum_{0\not=|n|\le M(Y_0)}C_{mn}a_n\Big|
    <\frac{2\eps}{|y^{1/2}K_{ir}(2\pi|m|y)|}
    \qquad\forall0\not=|m|\le M(Y_0).
  \end{align*}
\end{alg}
\begin{rem}\label{rem:y1}
  If we make a good choice for the height $y$ of the horocycle
  and take some $M_0\ge M(Y_0)$ such that $|y^{1/2}K_{ir}(2\pi|m|y)|\gg1$
  for all $0\not=|m|\le M_0$, the linear system of equations
  \eqref{eq:Hejhal} is well-conditioned for $0\not=|m|\le M_0$.
\end{rem}

\section{Finding eigenvalues}\label{Linear}

It remains to find the eigenvalues $\lambda$ for which the linear
system $C(\lambda)\vec a\simeq0$ has non-trivial solutions
$\vec a\not=0$, and for which the non-trivial solutions are independent of
the height $y$ of the horocycle.

To this end, we discretise the $\lambda$-axis into trial values
$0<\tilde\lambda_1<\tilde\lambda_2<\tilde\lambda_3<\ldots$
which lie sufficiently close together, and explore the neighbourhood
of each trial value for non-trivial solutions of
$C(\tilde\lambda+h)\vec a\simeq0$.
For the latter, we linearise with respect to the eigenvalue,
cf.\ \cite{Ruhe1973},
\begin{align*}
  C(\tilde\lambda+h)
  =C(\tilde\lambda)+hC'(\tilde\lambda)+\frac{h^2}2R(\tilde\lambda,h),
\end{align*}
where $R(\tilde\lambda,h)$ is a matrix whose norm is bounded by
\begin{align*}
  ||R(\tilde\lambda,h)||\le\sup_{|\xi|\le|h|}||C''(\tilde\lambda+h)||.
\end{align*}

Let $\tilde\lambda$ be fixed, $\eta$ be a small complex parameter, and let
$(h_\eta,\vec\alpha_\eta)_{|\tilde\lambda}$ be the non-trivial solutions of
\begin{align}\label{eq:perturbed}
  {C'}^{-1}(\tilde\lambda)\big(C(\tilde\lambda)
  +\frac{\eta^2}2R(\tilde\lambda,\eta)\big)\vec\alpha_\eta
  =-h_\eta\vec\alpha_\eta.
\end{align}
For $\eta=0$, we know
$(h,\vec\alpha)_{|\tilde\lambda}:=(h_0,\vec\alpha_0)_{|\tilde\lambda}$
numerically.
If $C'(\tilde\lambda)$ is regular, there are $2M(Y_0)$ matrix eigenvalues
of ${C'}^{-1}(\tilde\lambda)C(\tilde\lambda)$, and there are $2M(Y_0)$
generalized eigenvectors.

For $\eta\not=0$, we use perturbation theory to find
$(h_\eta,\vec\alpha_\eta)_{|\tilde\lambda}$.
\begin{conj}\label{conj:h}
  Under suitable conditions, in particular $|\eta|$ sufficiently small, we have
  \begin{align*}
    h_\eta=h+O(\eta^2) \quad \text{and} \quad
    \vec\alpha_\eta=\vec\alpha+O(\eta).
  \end{align*}
\end{conj}
Apart from numerical confirmations, it would be nice to have a proof
of Conjecture \ref{conj:h}.
However, it is not even evident that all solutions of \eqref{eq:perturbed}
have a Taylor expansion in $\eta$ around $0$.

If Conjecture \ref{conj:h} is true, it implies the following:
Using $\eta=h_\eta$ in \eqref{eq:perturbed}, we get
$C(\tilde\lambda+\eta)\alpha_\eta=0$ with $\eta=h+O(\eta^2)$
and $\vec\alpha_\eta=\vec\alpha+O(\eta)$.
For $|\eta|$ sufficiently small this can be solved for $\eta$ which
gives $\eta=h+O(h^2)$.
This allows to find numerical approximations,
$\lambda=\tilde\lambda+h+O(h^2)$, for the eigenvalues of
$C(\lambda)\vec a\simeq0$ in some neighbourhood of the trial value
$\tilde\lambda$.
The approximations can be refined iteratively.
This establishes:
\begin{thm}
  If Conjecture \ref{conj:h} holds true, and if the trial values
  $0<\tilde\lambda_1<\tilde\lambda_2<\ldots<\tilde\lambda_{\text{max}}$
  lie sufficiently close together,
  the linearisation in the eigenvalue allows to find {\em all} eigenvalues
  of $C(\lambda)\vec a\simeq0$ in the interval
  $[0,\tilde\lambda_{\text{max}}]$.
\end{thm}

\begin{rem}\label{rem:lambda}
  \begin{enumerate}
  \item
    Not each solution of
    $\big({C'}^{-1}(\tilde\lambda)C(\tilde\lambda)+hI\big)\vec\alpha=0$
    approximates a Maass form.
    Clearly, $|h|$ can be too large such that the linearisation becomes
    inappropriate.
    This is of no concern, since we do not need to consider Maass forms whose
    eigenvalues are further away than the next trial value $\tilde\lambda$.
  \item
    It can happen that a formal solution is {\em dependent} on $y$.
    This has to be checked explicitly by re-evaluating $C(\lambda)\vec a$ for
    different heights $y$ of the horocycle.
    Only if $C(\lambda)\vec a\simeq0$ holds for all $0<y<Y_0$, we have found a
    Maass cusp form.
    Typically, it is enough to (re-)evaluate for a few different, but well
    chosen values of $y$.
  \item
    The Laplacian on $X$ is an essentially self-adjoint operator.
    Hence, all eigenvalues $\lambda$ are real.
    Note, however, that the matrix $C(\tilde\lambda)$ is not hermite by
    construction.
    As such there is no guarantee that the linearisation of
    $C(\tilde\lambda+h)\vec a\simeq0$ results in real solutions.
    But, if it turns out that a solution becomes independent of $y$
    then the eigenvalue $\lambda=\tilde\lambda+h+O(h^2)$ becomes real, and
    the Maass form can be made real.
  \item
    It may happen that there is no eigenvalue $\lambda$ near a trial value
    $\tilde\lambda$.
    In this case, there are still $2M(Y_0)$ formal solutions of
    $\big({C'}^{-1}(\tilde\lambda)C(\tilde\lambda)+hI\big)\vec\alpha=0$, but
    $|h|$ is too large, or the formal solutions depend on the height $y$
    of the horocycle.
  \item\label{miss.1}
    According to Remark \ref{rem:y1}, a good choice for the height $y$ of the
    horocycle is cruical.
    In addition, we have to take care that $C'(\tilde\lambda)$ is regular
    and well-conditioned.
    However, it is hard to predict a good value for the height $y$ of the
    horocycle in advance.
    Moreover, a good choice of $y$ depends on $\lambda$.

    We fix some reasonable value of $y$ and search with this value
    of $y$ for Maass forms.
    Sometimes, our choice of $y$ may be good, and sometimes not.
    Whenever the choice of $y$ is not good, we may miss Maass forms.
    We compensate for this by taking another reasonable value of $y$ and
    run the algorithm again.
  \item
    It may happen that eigenvalues degenerate.
    In principle, this is no problem for our algorithm.
    However, the speed of numerical convergence and the numerically implied
    errors are strongly affected.
    Whenever a degenerated eigenvalue occurs, a good choice for the height $y$
    of the horocycle becomes even more important.
  \item
    One may argue that Maass cusp forms, if completely desymmetrised, are
    conjectured to be non-degenerated.
    However, the algorithm does not distinguish whether a degeneracy happens
    between Maass cusp forms, or whether the degeneracy is between a Maass
    cusp form and another formal solution of
    $\big({C'}^{-1}(\tilde\lambda)C(\tilde\lambda)+hI\big)\vec\alpha=0$.
    In particular, degeneracies with formal solutions happen frequently.
  \item\label{miss.2}
    In order to speed up the numerics, we do not take the trial values
    $0<\tilde\lambda_1<\tilde\lambda_2<\ldots<\tilde\lambda_{\text{max}}$
    to be sufficiently close together, and
    hence {\em intentionally} risk to oversee solutions.
    We compensate for this by using Weyl's average law and Turing bounds
    after the fact to figure out whether eigenvalues have been overseen.
    Any missing eigenvalues are eventually found with the aid of Algorithm
    \ref{alg:control}.
  \item
    For each trial value $\tilde\lambda$, we search for Maass forms in a
    neighbourhood.
    The neighbourhoods may overlap.
    Additionally, we may re-run the algorithm.
    For this reason, we typically find each Maass form more than once.
    Whenever we find a Maass form, we consider it only, if it does not match
    any previously found Maass form.
  \end{enumerate}
\end{rem}

\section{Turing bound}\label{Turing}

From Remark \ref{rem:lambda}\eqref{miss.1} and \ref{rem:lambda}\eqref{miss.2}
it is obvious that we may miss some eigenvalues.
In the current section, we use average Weyl's law and Turing bounds to
detect whether eigenvalues have been missed, and to figure out small intervals
where the missing eigenvalues are to be found.

Let $\mathcal{N}(t)=\#\{\lambda\ |\ 1/4\le\lambda\le t^2+1/4\}$ count the
number of eigenvalues which fall into the interval $[1/4,t^2+1/4]$.
Average Weyl's law reads
\begin{align*}
  \mathcal{N}(t)=\mathcal{M}(t)+S(t),
\end{align*}
where $\mathcal{M}(t)$ is a smooth approximation to $\mathcal{N}(t)$ such that
the average value of the Weyl remainder $S(t)$ becomes zero in the limit
$t\to\infty$,
\begin{align*}
  \langle S(t)\rangle:=\frac1t\int_0^tS(\tau)d\tau
  \overset{t\to\infty}{\longrightarrow}0.
\end{align*}

For Maass forms on certain hyperbolic surfaces, average Weyl's law has been
derived explicitly \cite{BookerStrombergsson,JorgensonSmajlovicThen2012}.
An important example is Weyl's average law on the modular surface.
\begin{thm}[Average Weyl's law on the modular surface
    \cite{BookerStrombergsson}]
  On the modular surface $X=\SL(2,\Z)\backslash\h$ we have for $t>0$:
  \begin{align*}
    \mathcal{M}(t)=\frac1{12}t^2
      -\frac{2t}\pi\log\frac{t}{e\sqrt{\pi/2}}-\frac{131}{144}.
  \end{align*}
\end{thm}
The proof is given in \cite{BookerStrombergsson}.
Another proof can be found in \cite{JorgensonSmajlovicThen2012}.

\begin{rem}
  For $X=\SL(2,\Z)\backslash\h$ there are no small eigenvalues $0<\lambda<1/4$,
  \cite{Roelcke1956,BookerStrombergssonVenkatesh2006}.
\end{rem}

Let $\mathcal{N}^{\text{num}}(t)$ count the number of numerically found
eigenvalues in the interval $[1/4,t^2+1/4]$.
The difference $\mathcal{N}^{\text{num}}(t)-\mathcal{N}(t)$ is a non-positive
integer whose absolute value counts the number of solutions which have been
overseen numerically.

As $\mathcal{N}(t)$ is unknown, we consider
$S^{\text{num}}(t)=\mathcal{N}^{\text{num}}(t)-\mathcal{M}(t)$, instead.
$S^{\text{num}}(t)$ is a fluctuating function around its average value,
$\mathcal{N}^{\text{num}}(t)-\mathcal{N}(t)$, see figure \ref{fig:S}.
\begin{figure}
  \includegraphics{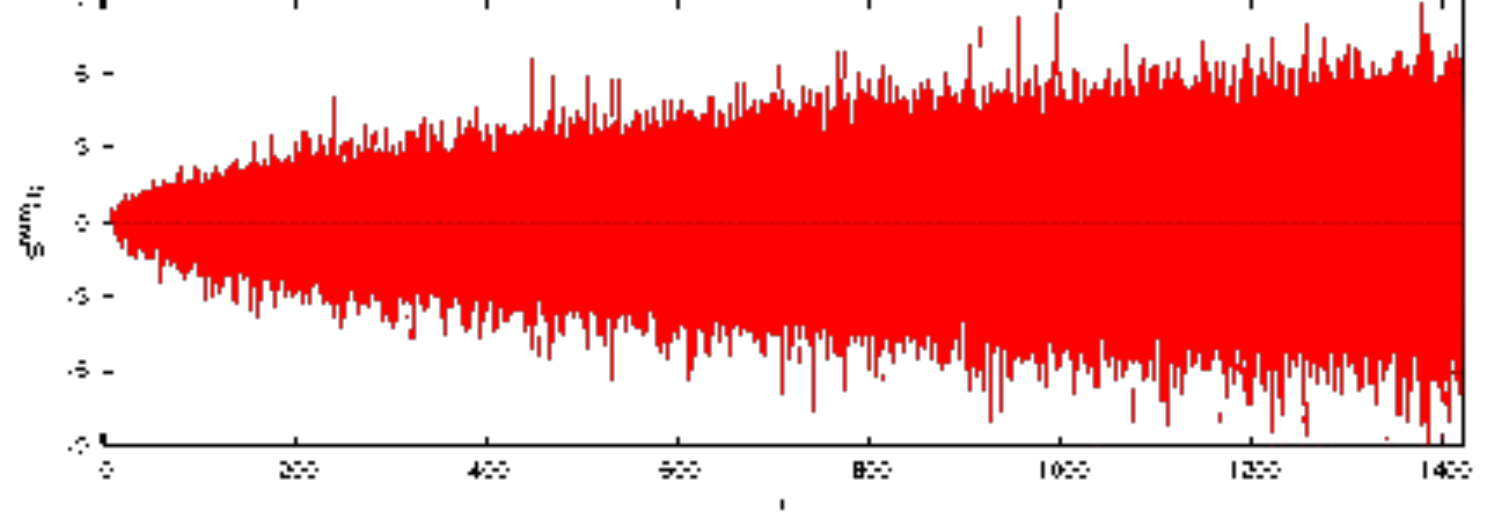}
  \caption{\label{fig:S}The fluctuations $S^{\text{num}}(t)$ for $\SL(2,\Z)$.}
\end{figure}
From the graph of
$\langle S^{\text{num}}(t)\rangle:=\frac1t\int_0^tS^{\text{num}}(\tau)d\tau$,
we can read off whether all eigenvalues have been found.
If all eigenvalues are found,
$\langle S^{\text{num}}(t)\rangle$ equals $\langle S(t)\rangle$ and tends
to zero in the limit of large $t$.
As soon as a solution gets overseen near $\lambda=r^2+1/4$ with $r\ge0$,
$\langle S^{\text{num}}(t)\rangle$ deviates significantly from
$\langle S(t)\rangle$ for $t>r$,
$\langle S^{\text{num}}(t)\rangle
\le\langle S(t)\rangle+\frac1t\int_r^t(-1)d\tau$.
This is visualised in figure \ref{fig:S-}, where we have intentionally removed
the eigenvalue $\lambda_{135916}=1/4+1300.0191255^2$.
From figure \ref{fig:S-}, we can read off that at least one eigenvalue is
missing and that the first eigenvalue which is missing is somewhere near
$\lambda=1/4+1300^2$.
\begin{figure}
  \includegraphics{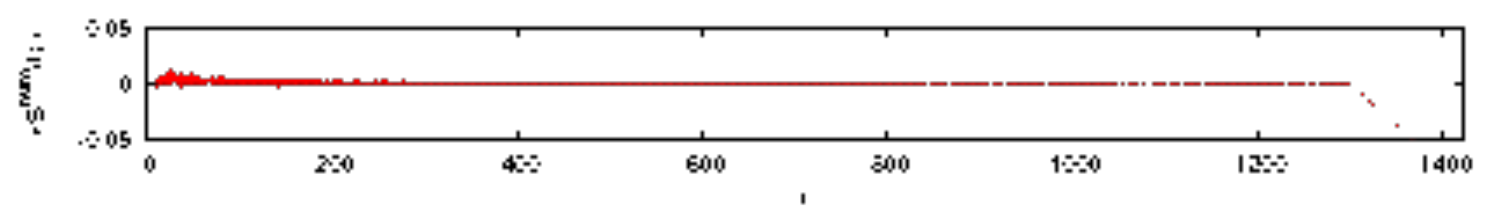}
  \caption{\label{fig:S-}Mean $\langle S^{\text{num}}(t)\rangle$, with the
    eigenvalue $\lambda_{135916}=1/4+1300.0191255^2$ removed.}
\end{figure}
\begin{figure}
  \includegraphics{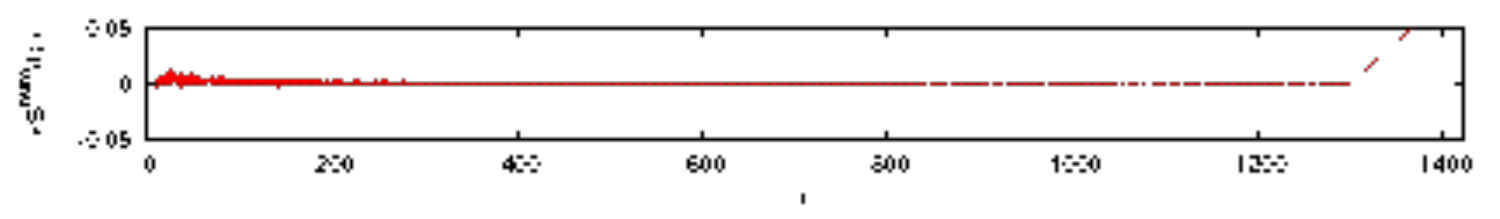}
  \caption{\label{fig:S+}Mean $\langle S^{\text{num}}(t)\rangle$, with the
    fake ``eigenvalue'' $\lambda_{\text{fake}}=1/4+1300^2$ inserted.}
\end{figure}

On the other hand, if we are in doubt whether a given list of eigenvalues is
consecutive, we can intentionally insert a fake ``eigenvalue'' near the end
of the list.
For instance, if we intentionally insert a fake ``eigenvalue''
$\lambda_{\text{fake}}=1/4+1300^2$ to the list of consecutive
eigenvalues, the graph of $\langle S^{\text{num}}(t)\rangle$ will clearly
show that there is one eigenvalue too much near $\lambda=1/4+1300^2$,
see figure \ref{fig:S+}.

If we know Turing bounds for $\langle S(t)\rangle$, the comparison of
$\langle S^{\text{num}}(t)\rangle$ against these bounds becomes conclusive.
For the modular group, Turing bounds have been derived explicitly
\cite{BookerStrombergsson}.
\begin{thm}[Turing bounds for the modular group \cite{BookerStrombergsson}]
  Consider the modular group $\Gamma=\SL(2,\Z)$, and define
  $E(t)=\big(1+6.59125/\log t\big)\big(\pi/(12\log t)\big)^2$.
  Then for $t>1$, we have the lower and upper Turing bounds
  \begin{align*}
    E_{\text{lower}}(t):=-2E(t)<\langle S(t)\rangle<E(t)=:E_{\text{upper}}(t).
  \end{align*}
\end{thm}
The proof is given in \cite{BookerStrombergsson}.

\begin{thm}[Turing's method \cite{Turing1953}]
  If we add a fake ``eigenvalue'', $\lambda_{\text{fake}}$,
  near the end of a list of eigenvalues, and if
  $\langle S^{\text{num}}(t)\rangle$ then exceeds the upper Turing bound,
  the list of eigenvalues is consecutive in the interval
  $1/4\le\lambda\le\lambda_{\text{fake}}$.
\end{thm}
The importance of this theorem is highlighted in
\cite{Booker2006,HejhalOdlyzko2012}.
For the proof, see \cite{Turing1953}.

\begin{thm}\label{thm:consecutive}
  Let a list of eigenvalues be given for which we compute
  $\langle S^{\text{num}}(t)\rangle$.
  Let lower and upper Turing bounds be given such that
  $-1/2<E_{\text{lower}}(t)<\langle S(t)\rangle<E_{\text{upper}}(t)<1/2$
  $\forall t\ge t_0$, where $t_0>0$.
  Let
  \begin{align*}
    T:=\min\{\tau\ge t_0
    \ |\ \langle S^{\text{num}}(\tau)\rangle\le E_{\text{lower}}(\tau)\}
  \end{align*}
  and
  \begin{align*}
    t:=T\max\big(0,\langle S^{\text{num}}(T)\rangle-E_{\text{upper}}(T)+1\big).
  \end{align*}

  Then, we have:
  The list of eigenvalues is consecutive for $1/4\le\lambda<t^2+1/4$, but
  there is an eigenvalue missing in $t^2+1/4\le\lambda\le T^2+1/4$.
\end{thm}
\begin{proof}
  Since there are countably many eigenvalues, the given finite list cannot
  be complete,
  $\langle S^{\text{num}}(t)\rangle\overset{t\to\infty}{\longrightarrow}
  \mathcal{N}^{\text{num}}(t)-\mathcal{N}(t)\le-1<E_{\text{lower}}(t)$.
  $T$ exists and we have
  $\langle S^{\text{num}}(T)\rangle\le E_{\text{lower}}(T)<\langle S(T)\rangle$,
  which implies that the list of eigenvalues is {\em not} consecutive for
  $1/4\le\lambda\le T^2+1/4$.

  If $\lambda=r^2+1/4$ is the first missed eigenvalue, we have
  $\langle S^{\text{num}}(T)\rangle
  \le\langle S(T)\rangle+\frac1T\int_r^T(-1)d\tau
  <E_{\text{upper}}(T)+\frac{r-T}T$.
  Hence,
  $\langle S^{\text{num}}(T)\rangle\ge E_{\text{upper}}(T)+\frac{t-T}T$
  implies $t<r$ and that all eigenvalues have been found for
  $1/4\le\lambda<t^2+1/4$ if $t\ge0$.
\end{proof}

\begin{alg}[The control algorithm]\label{alg:control}
  Let $\Lambda$ be a list of numerically found Maass cusp forms.
  Typically, $\Lambda$ is empty in the beginning.
  \begin{enumerate}
  \item Let all parameters take reasonable values.
  \item Determine $t$ and $T$ according to Theorem \ref{thm:consecutive}.
  \item\label{alg:control.3} Let
    $t^2+1/4=\tilde\lambda_1<\tilde\lambda_2<\ldots<\tilde\lambda_\nu=T^2+1/4$
    be trial values, and search for eigenvalues in the neighbourhood of each
    trial value.
  \item Use Algorithm \ref{alg:phase1} to compute the corresponding Maass
    cusp forms.
  \item For each found Maass form, compare whether it is already in the list
    $\Lambda$.
    If not, add it to the list.
  \item Let $t_{\text{old}}:=t$ and $T_{\text{old}}:=T$.
  \item Determine $t$ and $T$ according to Theorem \ref{thm:consecutive}.
  \item If $t$ is equal to $t_{\text{old}}$ and $T$ is equal to
    $T_{\text{old}}$, then increase $\nu$ and change all parameters to
    somewhat different reasonable values.
    Otherwise, decrease $\nu$ slightly in order to speed up the numerics.
  \item Continue with \ref{alg:control.3}
  \end{enumerate}
\end{alg}

Theoretically, Algorithm \ref{alg:control} runs forever and finds an ever
increasing list of consecutive Maass cusp forms.

For groups other than $\SL(2,\Z)$, Turing bounds have not been derived, yet.
It is possible to replace Turing bounds by heuristic bounds
which follow from {\em intentionally} removing and inserting eigenvalues near
the end of an almost consecutive list of eigenvalues, as was demonstrated in
figures \ref{fig:S-} and \ref{fig:S+}.
However, this requires user interaction and slows down the algorithm.
Results for certain moonshine groups are published in
\cite{JorgensonSmajlovicThen2012}.

\section{Fluctuations in the Weyl remainder}\label{Results}

Here, we present results for the modular group $\Gamma=\SL(2,\Z)$.
It is worth while to desymmetrise the Maass cusp forms, as this speeds up
their computation.
On the modular surface $X=\Gamma\backslash\h$, Maass cusp forms fall into
even, $f(-x+iy)=f(x+iy)$, and odd, $f(-x+iy)=-f(x+iy)$, eigenfunctions.
Letting Algorithm \ref{alg:control} run for some days, we have found
$80\,443$ consecutive even Maass cusp forms with $1/4\le\lambda<1421.98^2+1/4$,
but missed $1$ out of $81\,673$ with $1/4\le\lambda<1432.64^2+1/4$ by the
time when we interrupted the algorithm.
For the odd symmetry, the algorithm was faster.
In the same time, we have found $164\,232$ consecutive odd Maass cusp forms
with $1/4\le\lambda<2000^2+1/4$.
The amount of data allows us to investigate the asymptotic and statistic
properties of the fluctuations in the Weyl remainder $S(t)$, numerically.

From Rudnick \cite{Rudnick2005} we know that the number of eigenvalues
obeys a central limit theorem.
In particular, he considers smoothed windows of length $1/L$ centered
at points $t$, and examins the number $N_L(t)$ of $r_j$ that
lie within each such window.
In the limit $L\to\infty$ and $T\to\infty$ but $L=o(\log T)$, the
distribution of $N_L(t)$ as $t$ varies through $[T,2T]$ approaches
a Gaussian,
\begin{align*}
  \lim_{T\to\infty}\frac1T\meas\{t\in[T,2T]:
  \frac{N_L(t)-\overline{N_L}(t)}{\sigma_L}<x\}
  =\int_{-\infty}^xe^{-u^2/2}\frac{du}{\sqrt{2\pi}}.
\end{align*}
We ask whether the fluctuations in the Weyl remainder (without smoothing)
also obey a central limit theorem.

In view of figure \ref{fig:S}, we note that the magnitude of the remainder
fluctuations grows in $t$, and we may better consider a scaled version of
the Weyl remainder.
From Li and Sarnak \cite{LiSarnak2004} we know a lower bound on the asymptotic
growth rate of the remainder fluctuations,
\begin{align*}
  S(t)=\Omega\Big(\frac{t^{1/2}}{\log t}\exp\big(\tfrac12(\log\log t)^{5/17}
  \big)\Big).
\end{align*}
Figure \ref{fig:lower_bound} shows how close this lower bound
is to the true asymptotic growth rate.
Concerning an upper bound, nothing better than $S(t)=O(t/\log t)$ is known
analytically.
\begin{conj}\label{conj:upper-bound}
  On the modular surface $X=\SL(2,\Z)\backslash\h$, the asymptotic
  growth rate of the remainder fluctuations is bounded from above by
  \begin{align*}
    S(t)=o(\sqrt t).
  \end{align*}
\end{conj}
Numerical evidence for the conjecture is displayed in figure
\ref{fig:upper_bound}, where the magnitude of fluctuations of
$\frac1{\sqrt t}|S(t)|$ slightly decreases with $t$.
\begin{figure}
  \includegraphics{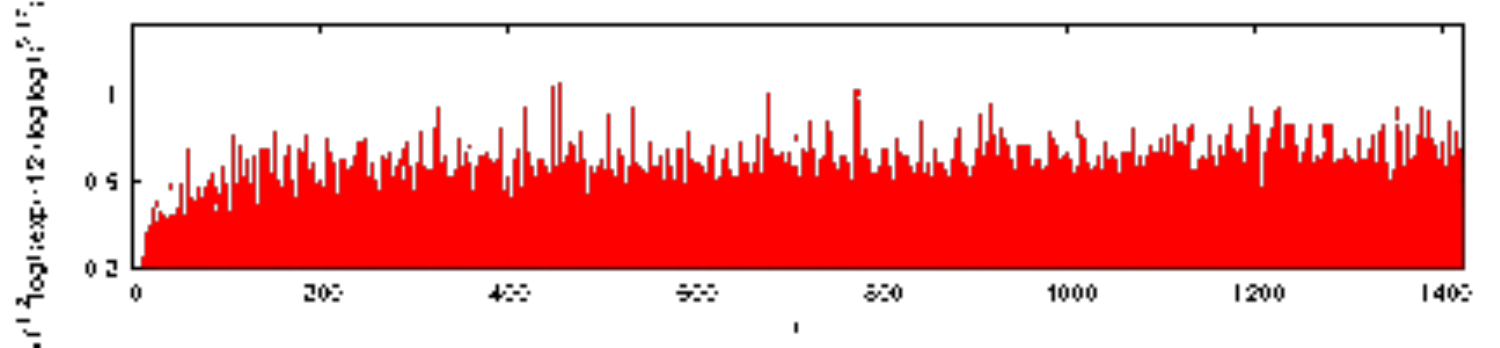}
  \caption{\label{fig:lower_bound}A graph of
    $\frac{\log t}{\sqrt t}\exp\big(-\tfrac12(\log\log t)^{5/17}\big)|S(t)|$
    whose magnitude of fluctuations slightly increases with $t$.
    This eventually lets the $\limsup$ of the graph diverge in the limit
    $t\to\infty$.
    The graph also shows how close the lower bound of Li and Sarnak comes
    to the true asymptotic growth rate.}
\end{figure}
\begin{figure}
  \includegraphics{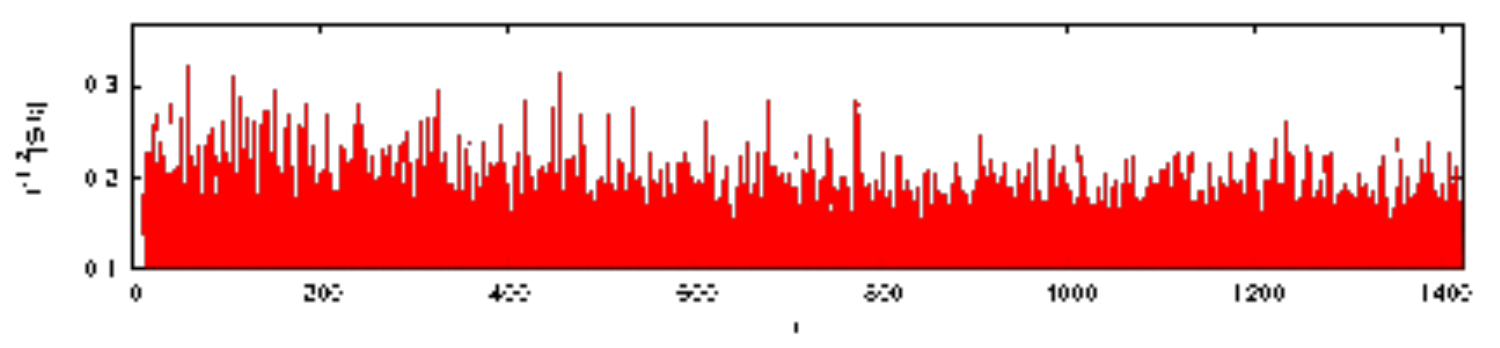}
  \caption{\label{fig:upper_bound}A graph of $\frac1{\sqrt t}|S(t)|$
    whose magnitude of fluctuations slightly decreases with $t$.
    This eventually lets the graph vanish in the limit $t\to\infty$.}
\end{figure}

\begin{conj}\label{conj:distr}
  Consider the fluctuations in the Weyl remainder $S(t)$ on the modular
  surface $X=\SL(2,\Z)\backslash\h$.
  If scaled by $\frac{\log\log t}{\sqrt t}$,
  the fluctuations obey a central limit theorem.
  For any ${\mathcal L}>0$, we have
  \begin{align*}
    \lim_{T\to\infty}\frac1{\mathcal L}\meas\{t\in[T,T+{\mathcal L}]:
    \frac{\log\log t}{\sqrt t}\frac{S(t)}\sigma<x\}
    =\int_{-\infty}^xe^{-u^2/2}\frac{du}{\sqrt{2\pi}},
  \end{align*}
  with $\sigma\approx0.140$.
\end{conj}
A histogram of the scaled distribution is shown in figure \ref{fig:S_distr}.
\begin{figure}
  \includegraphics{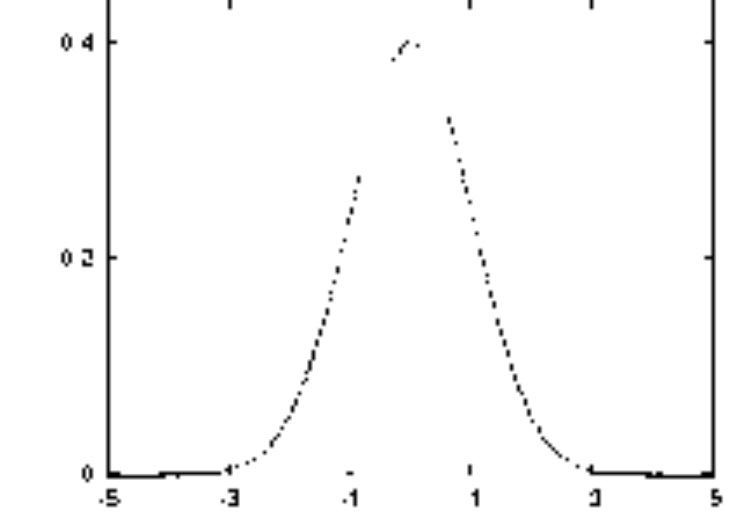}
  \caption{\label{fig:S_distr}Distribution of
    $\frac{\log\log t}{\sqrt t}\frac{S(t)}\sigma$ (solid curve)
    as $t$ varies through $[0,1421.98]$,
    in comparison with the standard Gaussian (dashed curve).}
\end{figure}
The scale factor $\frac{\log\log t}{\sqrt t}$ results from numerical
considerations.
Namely, if we examine $S(t)$ in the interval $[a,b]$, the mean and the
variance of the scaled fluctuations,
\begin{align*}
  &\mu_{[a,b]}=\frac1{b-a}\int_a^b\frac{\log\log t}{\sqrt t}S(t)dt,\\
  &\sigma_{[a,b]}^2=\frac1{b-a}\int_a^b\Big(\frac{\log\log t}{\sqrt t}S(t)
  -\mu_{[a,b]}\Big)^2dt,
\end{align*}
are numerically independent of $t$.
Quantitative results for $\mu_{[t-100,t+100]}$ and $\sigma_{[t-100,t+100]}$ in
dependence of $t$ are shown in figure \ref{fig:sigma}.
We find that the mean vanishes $\mu\approx0$ and that the standard
deviation is constant $\sigma\approx0.140$.
\begin{figure}
  \includegraphics{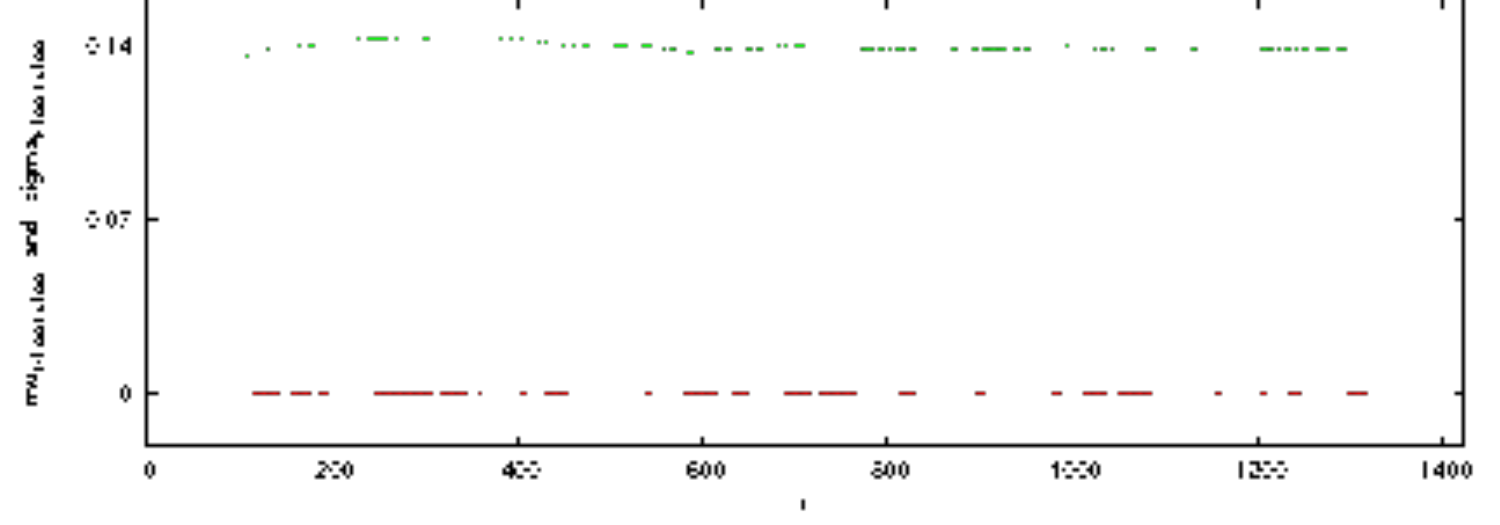}
  \caption{\label{fig:sigma}Plots of the mean $\mu_{[t-100,t+100]}$ and the
    standard deviation $\sigma_{[t-100,t+100]}$ in dependence of $t$.}
\end{figure}

By analogy with the law of the iterated logarithm, one might expect the
correct scalings for the distribution and extreme values to differ by a
factor of $\sqrt{\log\log t}$.
However, for the $S(t)$ of Riemann zeta, it's thought
\cite{FarmerGonekHughes2007} that they differ by about $\sqrt{\log t}$,
which is quite a bit larger.
We expect that the difference between the $S(t)$ of Maass form eigenvalues
and that of Riemann zeros comes from the level statistics.
The eigenvalues of the Laplacian on arithmetic surfaces
are expected to be Poisson distributed
\cite{BogomolnyGeorgeotGiannoniSchmit1992,BolteSteilSteiner1992},
whereas the Riemann zeros are expected to lie on the critical line and
show a distribution in resemblance to the eigenvalues of random matrices
of the Gaussian unitary ensemble \cite{Montgomery1973,RudnickSarnak1996}.

We believe that the iterated logarithm heuristic is closer to the truth.
If the Weyl remainder would result from a Wiener process, it would
strictly follow the law of the iterated logarithm and we would have
\begin{align*}
  \limsup_{t\to\infty}\Big(\frac{\log\log t}{2t}\Big)^{1/2}|S(t)|
  =\sigma.
\end{align*}
We have checked this numerically and find that the predicted scaling
of extreme values agrees with our numerical data.
The scaling also agrees with the upper bound of Conjecture
\ref{conj:upper-bound}.
However, the numerical values of the extrema are larger than predicted
by a factor of $2.5$, i.e.\ we find
\begin{align*}
  \limsup_{t\to\infty}\Big(\frac{\log\log t}{2t}\Big)^{1/2}|S(t)|
  \approx2.5\sigma,
\end{align*}
which indicates that the fluctuations in the Weyl remainder do not
exactly follow the law of the iterated lagarithm.

\end{document}